\documentclass[10pt]{amsart}
\usepackage[utf8]{inputenc}

\usepackage{amssymb}
\usepackage{amsmath}
\usepackage{amsthm}
\allowdisplaybreaks

\usepackage[left=2cm, right=2cm, top=3cm, bottom=3cm]{geometry}

\usepackage{graphicx}
\graphicspath{{figures/}}
\usepackage{enumitem}

\usepackage{array}

\newtheorem{theorem}{Theorem}[section]
\newtheorem*{theorem*}{Theorem}
\newtheorem{corollary}[theorem]{Corollary}
\newtheorem{definition}[theorem]{Definition}
\newtheorem{lemma}[theorem]{Lemma}

\newtheorem{example}[theorem]{Example}

\newcommand{\integrald}{\text{d}}
\newcommand{\KL}{\mathtt{KL}}

\author{Piotr Kamieński}
\title{A quantitative version of the theorem on Khintchine's constant}

\begin{document}


\begin{abstract}
 In the paper we provide measure estimates for the set of numbers whose sequence of products of continued fraction partial quotients $M_n = a_1 \ldots a_n$ has exponential growth with rate close to the one predicted by Khintchine's theorem, i.e. for which
 \begin{equation*}
  e^{(\kappa - T)n} \leqslant M_n \leqslant e^{(\kappa + T)n}
 \end{equation*}
 for a fixed $T > 0$ and all $n$ greater than some fixed integer $N$, where $e^\kappa = 2.685\ldots$ is the Khintchine constant. Choosing $N$ large enough the measure can be made arbitrarily close to full, for any given $T$. The bounds are not of asymptotic nature, but explicit in terms of the parameters involved. In the proof we compile several known result of large deviations theory, employing the cumulant method in particular. We also discuss the numerical values of the quantities involved.
\end{abstract}

\maketitle

\section{Motivation}\label{sec:motivation}

Diophantine (and Brjuno) numbers\footnote{see definition \ref{def:diophantine-number} for details} are commonly used in small divisors problems (\cite{arnold:1961-small-denominators-1, russmann:brjuno-paper, yoccoz:circle}). In KAM theory, for instance, if the frequency $\omega$ of an invariant torus is Diophantine then this torus survives once a perturbation is introduced - this happens for small enough perturbation parameter values and under the additional twist condition. The term ``small enough'', however, if specified precisely by the KAM-type theorem usually means ``smaller than some explicit formula depending on the Diophantine constant $C$ and exponent $\tau$''(as in e.g. \cite{delallave:kam-without}). The problem with this approach appears when we consider a family of tori with varying frequencies. Changing $\omega$ by replacing either its first few decimal digits or continued fraction partial quotients does not change $\tau$, but might decrease $C$ quite significantly, lowering the applicability threshold of the theorem through a multiplicative correction.

We propose an alternative to the Diophantine condition, what we call the \emph{Khintchine-L\'evy condition} (or $\KL-$condition for short) to account for this disadvantage. In the present paper we give the definition of the Khintchine-L\'evy numbers and prove that the set of all those numbers is generic in the sense of Lebesgue measure, as is the case with the set of Diophantine numbers. Specifically we provide explicit lower bounds on the measure of the set of $\KL$ numbers in terms of the parameters involved in their definition. In a parallel paper \cite{kamienski:2018-cohomological} we employ $\KL$ numbers to prove a small denominators result that is similar in nature to ones already obtained for Diophantine numbers. Khintchine-L\'evy numbers, however, have one advantage over Diophantine ones, namely they are less sensitive to the aforementioned changes in the initial part of the continued fraction. In the estimates in \cite{kamienski:2018-cohomological} such changes are reflected only through a minor additive correction.

We briefly introduce some notation before proceeding with the details. We will be working with irrational numbers $\omega \in \mathbb{X} := [0,1] \setminus \mathbb{Q}$ considered as a probability space with the Borel $\sigma-$algebra and either the Lebesgue measure $\lambda$ or the Gauss measure $\gamma$ given in terms of a density function\footnote{Note that in particular the two measures are absolutely continuous with respect to one another. In particular the terms ``Gauss almost all'' and ``Lebesgue almost all'' can be used interchangeably.}: $d\gamma(x) = {d\lambda(x) \over (1+x)\log 2}$. Expected value of a random variable w.r.t. a measure $\mu$ will be denoted by $\mathbb{E}_\mu$ and $A^c$ will denote the complement of a set $A \subset \mathbb{X}$.

Each $\omega \in \mathbb{X}$ has a unique infinite continued fraction expansion into a sequence of partial quotients $a_j~\in~\mathbb{N}_+, j = 1, 2, \ldots$\footnote{Note that we consider $a_0 = 0$, since we are in $\mathbb{X} \subset (0,1)$.}:
\begin{equation}
 \omega = [a_1, a_2, a_3, \ldots] = \cfrac{1}{a_1
              + \cfrac{1}{a_2
              + \cfrac{1}{a_3
              + \cfrac{1}{\ddots}}}}.
\end{equation}
The shift on the continued fraction expansion is a measurable transformation known as the \emph{Gauss map} $G : \mathbb{X} \mapsto \mathbb{X}$:
\begin{equation}
 G(\omega) = G([a_1, a_2, a_3, \ldots]) = [a_2, a_3, \ldots] = \left\{ {1 \over \omega} \right\}\footnote{$\{\cdot\}$ denoting the fractional part of a number. We also use $\lfloor \cdot \rfloor$ and $\lceil \cdot \rceil$ for the floor and ceiling functions, respectively.}.
\end{equation}
It preserves the Gauss measure and is ergodic with respect to that measure (\cite{ryll-nardzewski}). We also note that $a_n(\omega) = a_1(G^{n-1}(\omega))$. For $n \geqslant 1$ we denote
\begin{equation}\label{eq:notation-Mn-Xn-Sn}
 M_n(\omega) := a_1(\omega) \cdot \ldots \cdot a_n(\omega), \qquad X_n := \log a_n \qquad \mbox{and} \qquad S_n := X_1 + \ldots + X_n.
\end{equation}

The main motivation comes from the classical theorem on Khintchine's constant (\cite{khintchine:khintchine-constant-original-paper}). It tells us that for almost all $\omega \in \mathbb{X}$ the limit of ${1 \over n} S_n(\omega)$ exists, is finite and constant as a function of $\omega$ within said full measure set. We denote this limit as $\kappa$ and refer to $e^\kappa$ as \emph{Khintchine constant}\footnote{This is consistent with existing literature, where Khintchine's constant is defined as the pointwise a.e. limit of $\root{n}\of{M_n}$.}. One can observe that ${1 \over n} S_n(\omega)$ is actually the time-$n$ average of the test function $X_1$ along the orbit of $\omega$ under the action of $G$. Khintchine's theorem is thus a consequence of the Birkhoff pointwise ergodic theorem and $\kappa$ must be equal to the spatial average of $X_1$:
\begin{equation}\label{eq:expected-x1}
 \kappa = \int_{\mathbb{X}} \log a_1 \integrald \gamma = \int_0^1 \log ( \lfloor x^{-1} \rfloor ) {\integrald x \over (1+x) \log 2}.\footnote{The integral can be further reformulated into a well known series expression $\kappa = \sum_{r=1}^\infty \log_2 (r) \log (1 + (r(r+2))^{-1})$, which provides a numerical value of $\kappa \approx 0.9878$.}
\end{equation}

In probabilists' language Khintchine's theorem is actually a strong law of large numbers for the sequence of ``samples'' $(X_n)_{n=1}^\infty$. This result was later improved in the form of a plethora of limit theorems (see the monograph \cite[Chapter 3]{iosifescu-kraaikamp} and references therein for a survey). To the author's knowledge, however, all of the existing results are of asymptotic nature, but none provide exact estimates of the measure with explicitly computed constants. Our main result, theorem \ref{thm:main-theorem}, aims to fill this gap in.

It is also worth noting that the sequence of denominators $(q_n)_{n=1}^\infty$ of convergents\footnote{defined as the denominators of the reduced fraction obtained by truncating the continued fraction expansion at $a_n$} to $\omega$ has also been extensively studied in the literature. In \cite{kamienski:2018-cohomological} we discuss why this sequence is even more important from the point of view small denominators problems and KAM theory. Notable results include the analogue of Khintchine's theorem by Khintchine and L\'evy (\cite{levy:khintchine-levy}), its refinement by Philipp and Stackelberg in the form of a law of the iterated logarithm (\cite{philipp-stackelberg}) and further refinements by Ibragimov \cite{ibragimov} and Misevičius \cite{misevicius}, who obtained a central limit theorem with error bounds. We were, however, unable to prove the counterpart of theorem \ref{thm:main-theorem} for the sequence $(q_n)$ for technical reasons, which we discuss in the final section \ref{sec:conclusion}.

In section \ref{sec:notation-statement} we provide the reader with a precise definition of a Khintchine-L\'evy number in definition \ref{def:KLcondition} and later in theorem \ref{thm:main-theorem} we specify how far from full the measure of the set of these numbers is. In section \ref{sec:proof} we introduce all the necessary tools and combine them into a proof of this result, the most important one being the \emph{cumulant method} in theorem \ref{thm:largedeviations2}, which provides estimates for the tails of a r.v. given the estimates for its cumulants. In section~\ref{sec:incremented} we formulate and discuss the proof of a variation on theorem \ref{thm:main-theorem} for a slight modification of the sequence $M_n$. Section \ref{sec:properties} contains a brief comparison of the Khintchine-L\'evy numbers with Diophantine numbers. We conclude the paper with a brief practical analysis of the numerical values of parameters used along its course and some final remarks in sections \ref{sec:numerical} and~\ref{sec:conclusion}. The simple, but lengthy formulas are contained within appendix \ref{sec:appendix} for clarity.

\section{Main result}\label{sec:notation-statement}

The idea behind the Khintchine-L\'evy condition is the following. From Khintchine's theorem we can vaguely conclude that on a full measure set of $\omega$ the sequence $M_n(\omega)$ asymptotically exhibits exponential growth similar to $e^{\kappa n}$. We therefore conjecture that on a slightly smaller set, one whose measure is only close to full, the sequence $M_n(\omega)$ also exhibits exponential growth, but with slightly more relaxed requirements on its rate. Along the course of the paper we will learn that this is indeed the case, as stated in theorem \ref{thm:main-theorem}.

\begin{definition}[Khintchine-L\'evy condition]\label{def:KLcondition}
 We say that an irrational number $\omega$ is \emph{upper-$\KL$} with constants $T_+ > 0$ and $N \in \mathbb{N}$ if the following inequality holds for all $n \geqslant N$:
 \begin{equation}
  M_n(\omega) \leqslant e^{(\kappa + T_+)n}.
 \end{equation}
 Similarly, a number is \emph{lower-$\KL$} with constants $T_- > 0$ and $N \in \mathbb{N}$ if for all $n \geqslant N$ we have
 \begin{equation}
   e^{(\kappa - T_-)n} \leqslant M_n(\omega).
 \end{equation}

 We denote the sets formed by the numbers $\omega$ with the above properties by, respectively, $\KL^+(T_+,N)$ and $\KL^-(T_-,N)$. We also denote $\KL(T_-, T_+, N) := \KL^+(T_+, N) \cap \KL^-(T_-, N)$ and $\KL(T,N) := \KL(T,T,N)$ where $T > 0$.
 
 Also, for a given natural number $n$, we denote by $\KL^+_n(T)$ the set 
 \begin{equation}
  \KL^+_n(T) := \{ \omega \in \mathbb{X} : M_n(\omega) \leqslant e^{(\kappa + T)n} \}
 \end{equation}
 and similarly for $\KL^-$ and $\KL$.
 
 If a set of numbers $A \subset \mathbb{X}$ is of the form $A = \KL^\pm(T,N)$ for some $T$ and $N$ we will refer to it as a \emph{\mbox{(upper/lower-)Khintchine-L\'evy set}} or a \emph{$\KL$-set} for short.
\end{definition}

We are now ready to state the main result of this paper, which is in fact the aforementioned conjecture with all the necessary details accounted for.

\begin{theorem}[Estimates on the measure of $\KL$-sets]\label{thm:main-theorem}
 Let $N$ be a natural number and let $T$ be a positive real number. Denote
 \begin{equation}\label{eq:XiT}
  \Xi (T) = \exp \left( - {T^2 \over 2\left(128\bar r ^2 \bar\Lambda + \left(\left(16 \bar r \bar\Lambda \right)^{1/3} T \right) \right)^{3/2}} \right),
 \end{equation}
 where $\bar r$ and $\bar \Lambda$ are universal constants given in \eqref{eq:bar-r-definition} and \eqref{eq:bar-Lambda-definition}. Also denote $K = \left\lceil \sqrt{N} \right\rceil$. The lower bounds on the Gauss measures of Khintchine-L\'evy sets are given by
 \begin{equation}\label{eq:estimates-of-the-measure-of-KL-sets-1}
   \gamma \left(\KL^{\pm}(T, N)\right) \geqslant 1 - \left( \sum_{n=N}^{K^2 - 1} \Xi(T)^{\sqrt{n}} + {\Xi(T)^K \over 1 - \Xi(T)} \left( 2K + 1 + {4 \Xi(T) \over 1 - \Xi(T)} \right) \right).
 \end{equation}
 In particular for $N = K^2$ being a square of an integer we have
 \begin{equation}\label{eq:estimates-of-the-measure-of-KL-sets-2}
   \gamma \left(\KL^{\pm}(T, N)\right) \geqslant 1 - (1 - \Xi(T))^{-1} \left( 2\sqrt{N} + 1 + {4 \Xi(T) \over 1 - \Xi(T)} \right) \cdot \Xi(T)^{\sqrt{N}} .
 \end{equation}
\end{theorem}

Formulas \eqref{eq:XiT} and \eqref{eq:estimates-of-the-measure-of-KL-sets-2} imply in particular that regardless how small $T$ is we can still find an $N = N(T)$ such that the measure is $\varepsilon$-close to full for any fixed $\varepsilon > 0$. In section \ref{sec:numerical} we discuss the function $N(T)$ from a numerical point of view.

\section{Proof of theorem \ref{thm:main-theorem}}\label{sec:proof}

To estimate the measure of $\KL^\pm(T,N)$ from below is the same as to estimate the measure of its complement $\KL^\pm(T,N)^c$ from above. The complement, however, can be expressed as a sum of complements of $\KL_n(T)$:
\begin{equation}\label{eq:KL-intersection-complement}
  \KL^\pm(T,N)^c = \bigcup_{n=N}^\infty \KL^\pm_n(T)^c.
 \end{equation}
Our focus will therefore be centered on estimating $\gamma(\KL_n^\pm(T)^c)$ from above to use subadditivity of $\gamma$ in the end.\footnote{The sum in \eqref{eq:KL-intersection-complement} is not disjoint, but we are not concerned with the overlaps of the summands in the proof.}

\subsection{$\KL$-sets as tails of probability distributions}

To perform the proof of theorem \ref{thm:main-theorem} we first need to reformulate its statement in spirit of large deviations theory, we will mainly use the language of random variables $X_n$ and~$S_n$. Once this is done we will lay the framework of the proof out and fill in all the details in all the following subsections of this section.

First observe that $\mathbb{E}_\gamma X_j = \kappa$ for all $j$ and thus $\mathbb{E}_\gamma S_n = n\kappa$ - this is a consequence of the fact that $X_j = X_1 \circ G^{j-1}$ and the $G$-invariance of $\gamma$ (recall \eqref{eq:expected-x1}). Using this we can now write $\KL_n^+(T)^c$ in terms of centerings of $X_j$ and $S_n$:
\begin{equation}
  \KL^+_n(T)^c = \left\{ \omega \in \mathbb{X} : {1 \over n} \sum_{j=1}^n (X_j - \kappa) \geqslant T \right\} = \left\{ \omega \in \mathbb{X} :  S_n - n\kappa \geqslant Tn \right\}
\end{equation}
and similarly for $\KL_n^-(T)$.
This way estimating $\gamma(\KL_n^\pm(T))$ from below is the same as estimating the right/left tail of the centering of $S_n$ from above.

Our strategy will be the following. We first estimate the moments of $X_n$ in lemma \ref{lem:Xn-moment-estimates}. These moment estimates will allow us to use theorem \ref{thm:largedeviations1} to obtain estimates on the cumulants of $S_n$ and also of $S_n - n\kappa$\footnote{Shifting a random variable by a constant affects only the first cumulant, the only one we will not be concerned with.} - for this, however, we will need two additional assumptions on $X_n$: the \emph{$\varphi$-mixing assumption} and the \emph{Markov chain association assumption}. We introduce them in definitions \ref{def:phi-mixing} and \ref{def:markov-chain-association} and verify their validity for $X_n$ in lemmas \ref{lem:an-mixing} and \ref{lem:Xn-markov}. Once we have the cumulant estimates of $S_n - n\kappa$ we can estimate its tails - this is done with the help of theorem~\ref{thm:largedeviations2}.

Before we proceed we clarify what we exactly mean by moments and cumulants.
\begin{definition}
 Let $k \in \mathbb{N}_+$ and let $Y$ be a random variable on a probability space $(\mathbb{Y}, \mathcal{Y}, \mu)$. We define the $k$-th moment of $Y$ as $\mathbb{E}_\mu |Y|^k$ and the $k$-th cumulant of $Y$ as
 \begin{equation}
 \Gamma_k(Y) = {1 \over i^k} {d^k \over dt^k} \left(\log \left(\mathbb{E}_\mu e^{itY}\right) \right) \Big|_{t=0}.
 \end{equation}
 We will sometimes refer to $t \mapsto \log \left(\mathbb{E}_\mu e^{itY}\right)$ as the \emph{cumulant generating function}.
\end{definition}

\subsection{Moment estimates}

\begin{lemma}[Estimates of the moments of $X_n$]\label{lem:Xn-moment-estimates}
 The following estimates on the $k$-th moment of $X_n$ are valid for any $k \geqslant 2$ and $n \geqslant 1$:
 \begin{equation}
   \mathbb{E}_\gamma |X_n|^k \leqslant k! \cdot \bar r^k.
 \end{equation}
 Here
 \begin{equation}\label{eq:bar-r-definition}
  \bar r = \sqrt{3 \over 2\log 2} \approx 1.471.
 \end{equation}
\end{lemma}

\begin{proof}
 We will prove a stronger inequality, namely
 \begin{equation}
   \mathbb{E}_\gamma |X_n|^k \leqslant \bar r ^2 \cdot k!.
 \end{equation}
 In the formulation of the lemma, however, we decided to keep the (severe) exponential overestimation so that our result fits the framework of theorem \ref{thm:largedeviations1}.
 \\
 \begin{align}
 \begin{split}
   \mathbb{E}_\gamma |X_n|^k &= \mathbb{E}_\gamma \left| \log a_n \right|^k \stackrel{(\star)}{=} \mathbb{E}_\gamma |\log a_1|^k = \int_0^1 {|\log \lfloor x^{-1} \rfloor|^k \over (1+x) \log 2} dx \stackrel{(\star\star)}{\leqslant} \int_0^1 {|\log (x^{-1} - 1)|^k \over (1+x) \log 2} dx = \\
   &= \int_0^\infty {|\log^k y| dy \over (1+y)(2+y) \log 2} = \int_{-\infty}^\infty {|z|^k e^z dz \over (1+e^z)(2+e^z)\log 2} =\\
   &=\left( \int_{-\infty}^0 + \int_0^\infty \right) {|z|^k e^z dz \over (1+e^z)(2+e^z)\log 2}.
 \end{split}
 \end{align}
 Equality $(\star)$ is a consequence of $G$-invariance of $\gamma$, while in $(\star\star)$ we used the fact that $\lfloor x^{-1} \rfloor \geqslant x^{-1} - 1$ and that on the interval $(0,1)$ the function $|\log^k( \cdot )|$ is decreasing. In the equalities following $(\star\star)$ we simply substituted $x^{-1} - 1$ for $y$ and $y$ for $e^z$, respectively. After splitting the integral into the sum of two integrals we change the variables once again: on $[0, \infty)$ from $z$ to $u$ and on $(-\infty,0)$ from $z$ to $-u$. As a result we obtain
 \begin{align}
 \begin{split}
   \left( \int_{-\infty}^0 + \int_0^\infty \right) &{|z|^k e^z dz \over (1+e^z)(2+e^z)\log 2} = {1 \over \log 2} \left( \int_0^\infty {u^k e^u du \over (1+e^u)(2+e^u)} + \int_\infty^0 {u^k e^{-u} (-du) \over (1 + e^{-u})(2 + e^{-u})} \right) = \\
   &= {1 \over \log 2} \int_0^\infty u^k e^u \left[ {1 \over (1 + e^u)(2 + e^u)} + {1 \over (e^u + 1)(2e^u + 1)} \right] du = \\
   &= {1 \over \log 2} \int_0^\infty {3 u^k e^u du \over (2 + e^u)(1 + 2e^u)} \leqslant {3 \over 2\log 2} \int_0^\infty u^k e^{-u} du = {3 \over 2\log 2} \cdot k!.
 \end{split}
 \end{align}
 The last equality stems from the definition of the Euler gamma function and the fact that for integer arguments we have $\Gamma(k+1) = k!$.
\end{proof}

\subsection{Mixing properties of $(X_n)$}

There is a number of types of mixing for sequences of random variables (for a deeper insight see e.g. \cite{econometric-time-series} and references therein or \cite{bradley:2005}), the main idea behind all of them being the following: the further away from each other two random variables are in the sequence (in terms of the indexing number $n$) the closer they are to being independent. We will be primarily interested in the notion of \emph{$\varphi$-mixing}. However, a stronger property of $\psi$-mixing will also prove to be a useful tool.

\begin{definition}[$\varphi$-mixing sequence of r.vs, $\varphi$-mixing function and $\varphi$-mixing coefficients]\label{def:phi-mixing}
 Let $(Y_\nu)_{\nu=1}^\infty$ be a sequence of random variables on a probability space $(\mathbb{Y}, \mathcal{Y}, \mu)$. For indices $a \leqslant b \in \mathbb{N}_* \cup \{ \infty \}$ denote by $\sigma_a^b$ the $\sigma$-algebra generated by random variables $Y_\nu$ with $a \leqslant \nu \leqslant b$.
 We define the \emph{$\varphi$-mixing function of the sequence} $(Y_\nu)$ to be $\varphi : \mathbb{N}^2 \mapsto [0, 1]$ given by
 \begin{equation}
  \varphi(s,t) = \sup \left| \mu(B | A) - \mu(B) \right|
 \end{equation}
 where the supremum is taken over $A \in \sigma_1^s, B \in \sigma_t^\infty$ for which $\mu(A) > 0$.
 
 We define the \emph{$\varphi$-mixing coefficients of the sequence} $(Y_\nu)$ to be
 \begin{equation}
   \varphi_n = \sup_{k \in \mathbb{N}} \varphi(k, k+n).
 \end{equation}
 
 We say that the sequence $(Y_\nu)$ is \emph{$\varphi$-mixing} (w.r.t. $\mu$) if $\varphi_n \to 0$ as $n \to \infty$.
\end{definition}

The property of $\psi$-mixing is defined analogously, we alter only the mixing function in the definition:

\begin{definition}[$\psi$-mixing sequence of r.vs, $\psi$-mixing function and $\psi$-mixing coefficients]\label{def:psi-mixing}
  With the notations of definition \ref{def:phi-mixing} we define the \emph{$\psi$-mixing function of the sequence} $(Y_\nu)$ to be $\psi: \mathbb{N}^2 \mapsto [0, \infty]$ given by
 \begin{equation}
  \psi(s,t) = \sup \left| {\mu(A \cap B) \over \mu(A) \mu(B)} - 1 \right|
 \end{equation}
 where the supremum is taken over $A \in \sigma_1^s, B \in \sigma_t^\infty$ for which $\mu(A) \mu(B) > 0$.
 
 The \emph{$\psi$-mixing coefficients} $\psi_n$ are defined analogously to $\varphi_n$ in definition \ref{def:phi-mixing} and the sequence $(Y_\nu)$ is called $\psi$-mixing if they tend to $0$ with $n \to \infty$.
\end{definition}

The $\psi$-mixing property entails $\varphi$-mixing and additionally $\varphi_n \leqslant \psi_n/2$ (\cite{bradley:2005}). It turns out that the sequence $(a_n)$ enjoys the $\psi$-mixing property and the mixing coefficients decay at least exponentially fast:

\begin{lemma}[Quantitative estimates on the mixing coefficients of $(a_n)$, {\cite[Proposition 2.3.7]{iosifescu-kraaikamp}}]\label{lem:an-mixing}
 The coefficients $\psi_n$ of the sequence $(a_n)$ are bounded from above by $\psi_1 = 2\log 2 - 1 \approx 0.386$, $\psi_2 = {\pi^2 \log 2 \over 6} - 1 \approx 0.140$ and 
 \begin{equation}
  \psi_n \leqslant \psi_2 \lambda_0^{n-2}
 \end{equation}
 for all $n \geqslant 2$, where $\lambda_0$ is the Gauss-Kuzmin-Wirsing constant whose approximate value is $\lambda_0 \approx 0.304$.
\end{lemma}

Lemma \ref{lem:an-mixing} holds true also for the sequence $(X_n^\dag)_{n=1}^\infty$ (with exactly the same mixing coefficients). This is because the $\psi$-mixing property depends only on $\sigma-$algebras generated by the initial and tail parts of the sequence in question and these do not change upon composing the sequence with~a~bijective, measurable function (recall that $X_n = \log a_n$). This exponential decay will be useful for us in the technical results of subsection \ref{subsec:lambdafn}.

\subsection{Markov chain association}

\begin{definition}[Sequence of r.vs. associated to a Markov chain]\label{def:markov-chain-association}
  We say that a sequence of random variables $(Y_n)_{n=1}^\infty$ on a probability space $(\mathbb{Y}, \mathcal{Y}, \mu)$ is \emph{associated to a Markov chain} through a sequence of functions $(f_n)_{n=1}^\infty, f_n~:~\mathbb{R}~\mapsto~\mathbb{R}$ if
  \begin{equation}
   Y_n = f_n ( \xi_n )
  \end{equation}
  for a Markov chain $(\xi_n)_{n=1}^\infty$.
\end{definition}

\begin{lemma}\label{lem:Xn-markov}
  The sequence $(X_n)_{n=1}^\infty$ is associated to the Markov chain
  \begin{equation}\label{eq:sn-definition}
    s_n = [a_n, a_{n-1}, \ldots, a_1] = \cfrac{1}{a_n
              + \cfrac{1}{a_{n-1}
              + \cfrac{1}{\ddots
              + \cfrac{1}{a_1}}}}.
  \end{equation}
  through the sequence of functions $(f_n)_{n=1}^\infty$ given by
  \begin{equation}\label{eq:markov-chain-associating-functions}
   f_1(\xi) = \ldots = f_n(\xi) = \ldots = \log \lfloor \xi^{-1} \rfloor.
 \end{equation}
\end{lemma}

\begin{proof}
 Equality $X_n = \log \left\lfloor \left( s_n \right)^{-1} \right\rfloor$ is a direct consequence of $a_n < s_n^{-1} < a_n+1$, which stems from the definition of $s_n$. We thus have to prove that $(s_n)_{n=1}^\infty$ is indeed a Markov chain. The definition of a Markov chain requires a choice of probability (in our case a natural one would be to choose $\gamma$). However, $(s_n)$ is a Markov chain for any probability (for which the definition of a Markov chain makes sense). Observe that once the chain is at some state $\bar s \in \mathbb{Q}$ we can uniquely determine all its past states through the shift $G$. This way any conditional probability under the condition of all past states being fixed is actually the conditional probability under the condition of just the previous state being fixed, provided that these probabilities are nonzero, which is the case for $\gamma$.
\end{proof}

\subsection{The $\Lambda(f,n)$ quantity}\label{subsec:lambdafn}

We now have almost all the necessary tools to proceed with the estimation of the cumulants. We need, however, to define and study one more quantity - $\Lambda_n(f,u)$. Its nature is purely technical, but it will become crucial for us in the formulation of theorem \ref{thm:largedeviations1}.

\begin{definition}[The $\Lambda_n(f,u)$ quantity]\label{def:Lambda}
  Let $f : \mathbb{N}^2 \mapsto \mathbb{R}$ be a function and let $u \neq 0$. We define $\Lambda_n(f, u)$ to be
  \begin{equation}
    \Lambda_n(f, u) = \max \left\{ 1, \max_{1 \leqslant s \leqslant n} \sum_{t=s}^n f(s,t)^{1/u} \right\}.
  \end{equation}
\end{definition}

We will be primarily interested in $\Lambda_n(\varphi,2)$, where $\varphi$ is the $\varphi$-mixing function of the sequence $(X_n)$. Again, upper bounds on this quantity will turn out to be essential for us.

\begin{lemma}[Estimates on $\Lambda_n(\varphi,2)$ for the sequence $(X_n)$]\label{lem:Lambda-estimates}
 If $\varphi$ is the $\varphi$-mixing function of the sequence $(X_n)$ then the following inequality holds:
 \begin{equation}
  \Lambda_n(\varphi,2) \leqslant \bar\Lambda,
 \end{equation}
 where $\bar\Lambda$ is a universal constant given by
 \begin{equation}\label{eq:bar-Lambda-definition}
   \bar\Lambda = 1 + \left(\log 2 - {1 \over 2} \right)^{1/2} + \left( {\pi^2 \log 2 \over 12} - {1 \over 2} \right)^{1/2} \cdot {1 \over 1 - \lambda_0^{1/2}} \approx 2.029
 \end{equation}
 with $\lambda_0$ being the Gauss-Kuzmin-Wirsing constant.
\end{lemma}

For the reader acquainted with various types of mixing and metrical theory of continued fractions it may have appeared that we use the $\varphi$-mixing property with regard to $(a_n)$ (and $(X_n)$) unnecessarily, as these sequences enjoy the stronger property of $\psi$-mixing and therefore might be suitable for large deviation theorems which produce better estimates. This is not the case, however, as these theorems employ the eponymous $\Lambda_n(\cdot, \cdot)$, which in turn depends on $f(s,s)$ which may turn out to be infinite in the $f = \psi$ case. Before we proceed with the proof of lemma \ref{lem:Lambda-estimates} we clarify this subtlety in the following
\begin{example}\label{ex:psi0-infinite}
 Suppose that $(Y_n)$ is a sequence of r.vs. such that the $\sigma-$algebra generated by $Y_s$ admits sets of arbitrarily small measure for some $s$. Let $(C_n)$ be a sequence of sets in this $\sigma-$algebra whose measures $\mu(C_n)$ decrease to $0$. This $\sigma-$algebra is contained in both $\sigma_1^s$ and $\sigma_s^\infty$. We therefore have
 \begin{equation}
   \psi(s,s) \geqslant \sup_n \left| {\mu(C_n \cap C_n) \over \mu(C_n) \mu(C_n)} - 1 \right| = \sup_n \left|{1 \over \mu(C_n)} - 1\right| = \infty
 \end{equation}
\end{example}

The phenomenon described above does not appear if we use $\varphi$-mixing instead. Note that in definitions \ref{def:phi-mixing} and \ref{def:psi-mixing} we used $[0,1]$ and $[0,\infty]$ as the codomains for $\varphi$ and $\psi$, respectively. This is because $1$ is a natural upper bound for $\varphi(\cdot, \cdot)$ since we can estimate $|\mu(B|A) - \mu(B)| \leqslant \max\left\{ \mu(B | A), \mu(B) \right\} \leqslant 1$.

\begin{proof}[Proof of lemma \ref{lem:Lambda-estimates}]
  We will estimate the sum $\sum_{t=s}^n \varphi(s,t)^{1/2}$ using the $\psi$-mixing coefficients and lemma \ref{lem:an-mixing}. We will, however, take into account what has been said in example \ref{ex:psi0-infinite} and majorize all the terms in the sum except the first one, for which we use $\varphi(s,s) \leqslant \varphi_0 \leqslant 1$.
  
  We employ the bounds of lemma \ref{lem:an-mixing} for the remaining terms:
  \begin{align}
  \begin{split}
    \sum_{t=s}^n \varphi(s,t)^{1/2} &\leqslant \sum_{t=s}^n \varphi_{t-s}^{1/2} \leqslant \varphi_0^{1/2} + \sum_{t=1}^{n-s} (\psi_t/2)^{1/2} \leqslant \varphi_0^{1/2} + \sum_{t=1}^\infty (\psi_t/2)^{1/2} \leqslant \\ &\leqslant 1 + {\sqrt{2} \over 2} \left[ \psi_1^{1/2} + \psi_2^{1/2} \left( \sum_{t=0}^\infty \left(\lambda_0^{1/2}\right)^t \right) \right] = \bar\Lambda.
  \end{split}
  \end{align}
  Both the sum in curly brackets in \eqref{eq:bar-Lambda-definition} and the number $1$ are bounded from above by $\bar\Lambda$, which concludes the proof.
\end{proof}

\subsection{Estimating the cumulants of the centered sum}

We first state the abstract theorem that will allow us to pass from estimates on the moments of $X_n$ to estimates on the cumulants of $S_n$.

\begin{theorem}[Moment estimates imply cumulant estimates for the sum, {\cite[Theorem 4.21]{large-deviations-book}}]\label{thm:largedeviations1}
  Let $(Y_n)_{n=1}^\infty$ be a sequence of random variables defined on a probability space $(\mathbb{Y}, \mathcal{Y}, \mu)$ and denote
  \begin{equation}
   W_n = Y_1 + \ldots + Y_n.
 \end{equation}
 Assume that the sequence $(Y_n)$ is associated to some Markov chain and that it is $\varphi$-mixing. Assume also that it satisfies the following moment estimate:
 \begin{equation}\label{eq:abstract-moment-estimates}
  \mathbb{E}_\mu |Y_n|^k \leqslant (k!)^{1+\gamma_1} H_1^k
 \end{equation}
 for some constants $\gamma_1 \geqslant 0$ and $H_1 > 0$ and all integers $k \geqslant 2$ and $n \geqslant 1$. Then for each $k \geqslant 2, n \geqslant 1$ and $\delta > 0$ the following cumulant estimate is valid for $W_n$:
 \begin{equation}\label{eq:largedeviations1-estimate}
   |\Gamma_k(W_n)| \leqslant (k!)^{2+\gamma_1} \cdot 8^{k-1} \cdot H_1^k \cdot \lceil 1+\delta \rceil^{(1+\gamma_1)k} \cdot (\Lambda_n(\varphi,1+1/\delta))^{k-1} \cdot n,
 \end{equation}
 where $\Gamma_k$ are taken with respect to $\mu$.
\end{theorem}

Let us now apply theorem \ref{thm:largedeviations1} with $Y_n = X_n$. Its assumptions are verified with $\gamma_1 = 0$ and $H_1 = \bar r$ (lemma \ref{lem:Xn-moment-estimates}). Choosing $\delta = 1$ and applying the estimates on $\Lambda_n(\varphi,2)$ (lemma \ref{lem:Lambda-estimates}) we arrive at the following

\begin{theorem}[Cumulant estimates for $S_n$]\label{thm-largedeviations1-applied}
 For any $k \geqslant 2$ and $n \geqslant 1$ the $k$-th cumulant of $S_n$ is bounded by
 \begin{equation}
   |\Gamma_k(S_n)| \leqslant \left( {k! \over 2} \right)^2 \cdot (16\bar r \bar\Lambda)^{k-2} \cdot 128 \bar r ^2 \bar\Lambda \cdot n.
 \end{equation}
\end{theorem}

Theorem \ref{thm-largedeviations1-applied} holds also if we replace $S_n$ with its centering $S_n - n\kappa$ since shifting a random variable by a constant does not affect its cumulants of order $k \geqslant 2$. We will use this simple observation in what follows.

\subsection{Estimating the tails of the centered sum}

We now turn to estimating the tails of $S_n - n\kappa$. Once again we begin by stating the abstract large deviations theorem.

\begin{theorem}[Cumulant estimates imply tail estimates, {\cite[Lemma 2.4]{large-deviations-book}}, \cite{bentkus-rudzkis-1980}]\label{thm:largedeviations2}
  Let $W$ be a centered\footnote{i.e. $\mathbb{E}_\mu W = 0$} random variable defined on a probability space $(\mathbb{Y}, \mathcal{Y}, \mu)$. Assume there exist constants $\gamma_2 \geqslant 0, H > 0$ and $\bar\Delta > 0$ such that for all integers $k \geqslant 2$ we have
  \begin{equation}
   |\Gamma_k(W)| \leqslant \left( {k! \over 2} \right)^{1 + \gamma_2} {H \over \bar{\Delta}^{k-2}}.
  \end{equation}
  Then for all $x \geqslant 0$ the following inequality is valid:
  \begin{equation}
   \mu(\pm W \geqslant x) \leqslant \exp \left( - {x^2 \over 2\left(H + \left(x/\bar{\Delta}^{1/(1+2\gamma_2)}\right) \right)^{(1+2\gamma_2)/(1+\gamma_2)}} \right).
  \end{equation}
  Here $\Gamma_k$ denotes the cumulant taken w.r.t. $\mu$, while the notation $\pm W$ indicates that the inequality holds both for $W$ and $-W$ .
\end{theorem}

We may now plug the results of theorem \ref{thm-largedeviations1-applied} for $W = S_n - n\kappa$ into theorem \ref{thm:largedeviations2}. Its assumptions are verified for measure $\mu = \gamma$ and constants $\gamma_2 = 1, \bar\Delta = \left(16 \bar r \bar\Lambda \right)^{-1}$ and $H = 128 \bar r ^2 \bar\Lambda \cdot n$.

\begin{theorem}[Tail estimates for $S_n - n\kappa$]\label{thm:largedeviations2-applied}
 For any $n \geqslant 1$ and $x > 0$ the following tail estimate holds for $S_n - n\kappa$:
 \begin{equation}
   \gamma(\pm (S_n - n\kappa) \geqslant x) \leqslant \exp \left( - {x^2 \over 2\left(128 \bar r ^2 \bar\Lambda \cdot n + \left(\left(16 \bar r \bar\Lambda \right)^{1/3} \cdot x \right) \right)^{3/2}} \right).
 \end{equation}
\end{theorem}

We now combine the results of this section altogether to obtain the desired estimates on the measure of $\KL$-sets.

\begin{proof}[Proof of theorem \ref{thm:main-theorem}]
 Rewriting the estimates of theorem \ref{thm:largedeviations2-applied} with $x = nT$ we arrive at
 \begin{equation}\label{eq:largedeviations2-applied-rewritten}
  \gamma(\KL_n^\pm(T)) = \gamma(\pm (S_n - n\kappa) \geqslant n T) \leqslant \exp \left( - \sqrt{n} \cdot {T^2 \over 2\left(128 \bar r ^2 \bar\Lambda + \left(\left(16 \bar r \bar\Lambda \right)^{1/3} T \right) \right)^{3/2}} \right) = \Xi(T)^{\sqrt{n}}.
 \end{equation}
 The final estimates \eqref{eq:estimates-of-the-measure-of-KL-sets-1} and \eqref{eq:estimates-of-the-measure-of-KL-sets-2} stem from \eqref{eq:KL-intersection-complement}, the subadditivity of $\gamma$ and the estimates for the sum of terms of the form $e^{-\alpha \sqrt{n}}$ over $n \geqslant N$ for $\alpha > 0$ contained in lemma \ref{lem:exp-sqrt-sum} in appendix \ref{sec:appendix}.
\end{proof}

\section{The case of incremented partial quotients}\label{sec:incremented}

Theorem \ref{thm:main-theorem} is not limited for application only to the sequence $(M_n)$, one can also use it for other sequences for which a counterpart of Khintchine's theorem on Khintchine constant holds. We demonstrate it for the sequence of products of incremented partial quotients $(M_n')_{n=1}^\infty$:
\begin{equation}
 M_n' := (1 + a_1)\ldots(1 + a_n).
\end{equation}
We choose $(M_n')$ among other sequences for this purpose since it provides an upper bound for the sequence of denominators of convergents $(q_n)$, similarly to $(M_n)$, which provides a lower bound. This proves useful in the small divisors estimates that we perform in \cite{kamienski:2018-cohomological}.

We begin by introducing the notations that are a counterpart of \eqref{eq:notation-Mn-Xn-Sn}:
\begin{equation}\label{eq:notation-Mn-Xn-Sn-prime}
 X_n' := \log (1 + a_n) \qquad \text{and} \qquad S_n' := X_1' + \ldots + X_n'
\end{equation}
for $n \geqslant 1$. By Birkhoff's pointwise ergodic theorem the sequence ${1 \over n} S_n'(\omega)$ tends to a constant almost everywhere just like in the theorem on Khintchine's constant. This time, however, the test function is $X_1'$, therefore ${1 \over n} S_n ' \to \kappa'$ with
\begin{equation}
 \kappa' = \int_\mathbb{X} X_1' \integrald \gamma = \int_0^1 \log ( 1 + \lfloor x^{-1} \rfloor) {dx \over (1+x)\log 2} \approx 1.410.
\end{equation}
The Khintchine-L\'evy sets are thus defined as
\begin{equation}
 \KL'^+(T,N) := \{ \omega \in \mathbb{X} : M_n'(\omega) \leqslant e^{(\kappa' + T)n} \mbox{ for all } n \geqslant N \}
\end{equation}
for $T > 0$ and $N \in \mathbb{N}$. The sets $\KL'^-(T,N), \KL'(T,N)$ and $\KL_n'^\pm(T)$ are defined analogously to definition \ref{def:KLcondition}. Theorem \ref{thm:main-theorem} for $\KL'$-sets reads
\begin{theorem}[Estimates on the measure of $\KL'$-sets]\label{thm:main-theorem-Mn-prime}
 Let $N$ be a natural number and let $T$ be a positive real number. Denote
 \begin{equation}\label{eq:bar-r-prime-definition}
  \bar r' := \sqrt{\eta(3) \over \log 2} \approx 1.140,
 \end{equation}
 where $\eta$ is the Dirichlet $\eta$ function: $\eta(s) := \sum_{n=1}^\infty (-1)^{n-1} n^{-s}$. Define $\Xi'(T)$ as in \eqref{eq:XiT}, but with $\bar r'$ in place of $\bar r$. With the notations of theorem \ref{thm:main-theorem} the estimates on the measures of $\KL'^\pm(T,N)$ are the same as in \eqref{eq:estimates-of-the-measure-of-KL-sets-1} and \eqref{eq:estimates-of-the-measure-of-KL-sets-2}, but with $\Xi'(T)$ in place of $\Xi(T)$.
\end{theorem}

\begin{proof}
 For the theorem to be proven one needs lemma \ref{lem:Xn-moment-estimates} to hold for the sequence $(X_n')$ along with the equality of averages $\mathbb{E}_\gamma X_n' = \kappa'$ for all $n \in \mathbb{N}$. The claim on averages follows from the $G$ invariance of $\gamma$, as was the case with the sequence $(X_n)$: we have $X_j' = X_1' \circ G^{j-1}$ for all $j$. The cumulant estimates of theorem \ref{thm-largedeviations1-applied} depend only on the constants in the moment estimates and the value of $\bar \Lambda$. The latter stems in turn from the mixing coefficients of the sequence in question, which do not change when we switch from $(X_n)$ to $(X_n')$. The Markov chain association assumption also holds for $(X_n')$, only for a different sequence of functions: $\xi \mapsto \log \lfloor \xi^{-1} \rfloor$ changes to $\xi \mapsto \log ( 1 + \lfloor \xi^{-1} \rfloor )$ in \eqref{eq:markov-chain-associating-functions}. With that the whole proof forms a food chain that feeds on the moment estimates, which read
 \begin{align}
 \begin{split}
  \mathbb{E}_\gamma |X_n'|^k &= \mathbb{E}_\gamma |\log(1 + a_n)|^k = \mathbb{E}_\gamma |\log(1+a_1)|^k = \int_0^1 {|\log(1+\lfloor x^{-1} \rfloor)|^k \over (1+x) \log 2} \integrald x \leqslant \int_0^1 {|\log(x^{-1})|^k \over (1+x) \log 2} \integrald x = \\ &=\int_1^\infty {|\log^k y| \over (y^2 + y)\log 2} \integrald y = \int_0^\infty {|z|^k e^z \over (e^{2z} + e^z) \log 2} \integrald z = \int_0^\infty {z^k dz \over (1 + e^z) \log 2} = {\eta(k+1) \over \log 2} \cdot k! \leqslant {\eta(3) \over \log 2} \cdot k!.
 \end{split}
 \end{align}
 The changes of variables used along the way are $x^{-1} = y$ and $y = e^z$. We also employed the standard formulas for the Dirichlet $\eta$ function: \begin{equation}\label{eq:dirichlet-eta}
   \eta(k+1) = {1 \over k!} \int_0^\infty {z^k dz \over 1 + e^z} = \sum_{n=1}^\infty {(-1)^{n-1} \over n^{k+1}}
 \end{equation}
 and the fact that it is decreasing with $k$ so that $\eta(k+1) \leqslant \eta(3)$ for $k \geqslant 2.$
\end{proof}

\section{Properties of Khintchine-L\'evy numbers}\label{sec:properties}

In this section we briefly compare Khintchine-L\'evy numbers with Diophantine numbers. We begin by recalling the definition of the latter along with a few well-known properties.

\begin{definition}[Diophantine number]\label{def:diophantine-number}
 Let $\tau \geqslant 1$ and $C > 0$. We say that a real number $\omega$ is \emph{$(C,\tau)$-Diophantine} if the inequality
 \begin{equation}\label{eq:dioph-condition}
  |q\omega - p| \geqslant {C \over |q|^\tau}
 \end{equation}
 holds for all integers $p$ and $q$ with $q \neq 0$. A number is called \emph{Diophantine} if it is $(C, \tau)$-Diophantine for some $C > 0$ and $\tau \geqslant 1$.
\end{definition}

We also have the following characterization of Diophanticity in terms of the continued fraction expansion:
\begin{lemma}[Diophanticity in terms of the continued fraction expansion]\label{lem:diophanticity-in-terms-of-continued-fraction}
 If an irrational number $\omega$ is $(C,\tau)$-Diophantine with $C>0$ and $\tau \geqslant 1$ then its partial quotients can be estimated by
 \begin{equation}\label{eq:diophanticity-in-terms-of-continued-fraction}
  a_{n+1} \leqslant C^{-1} q_n^{\tau - 1}.
 \end{equation}
 Conversely, an estimate as in \eqref{eq:diophanticity-in-terms-of-continued-fraction} for all $n \geqslant 0$ results in $\omega$ being $(C/(1+2C), \tau)$-Diophantine.
\end{lemma}

\begin{proof}
 If a number is $(C,\tau)$-Diophantine we have ${1 \over q_{n+1}} > |q_n\omega - p_n| \geqslant {C \over q_n^\tau}$ which gives ${q_{n+1} \over q_n} < C^{-1} q_n^{\tau - 1}$ and this gives \eqref{eq:diophanticity-in-terms-of-continued-fraction} since ${q_{n+1} \over q_n} > a_{n+1}$.
 
 For the reverse implication fix $n$ and suppose we have $q_n \leqslant q < q_{n+1}$. We have that $|q_n \omega - p_n| \leqslant |q \omega - p|$ for such $q$ and any $p$ and also ${D \over q^\tau} \leqslant {D \over q_n^\tau}$ for any $D > 0$. Therefore it suffices to show that ${D \over q_n^\tau} \leqslant |q_n \omega - p_n|$ with $D = {C \over 1 + 2C}$. Assuming \eqref{eq:diophanticity-in-terms-of-continued-fraction} we have, however,
 \begin{align}
 \begin{split}
  |q_n \omega - p_n| &> {1 \over q_{n+1} + q_n} = {1 \over a_{n+1} q_n + q_{n-1} + q_{n}} > {1 \over (a_{n+1} + 2) q_n} > {1 \over (C^{-1} q_n^{\tau-1} + 2) q_n} = \\
  &= {1 \over C^{-1} q_n^\tau + 2 q_n} \geqslant {1 \over (C^{-1} + 2) q_n^\tau} = { {C \over 1+ 2C} \over q_n^\tau }.
 \end{split}
 \end{align}
 The above reasoning works regardless of the choice of $n$, therefore the proof is concluded.
\end{proof}

The denominators $(q_n)$ satisfy a recurrence relation
\begin{equation}\label{eq:qn-recurrence}
 q_n = a_n q_{n-1} + q_{n-2}, \qquad q_{-1} = 0, q_{-2} = 1,
\end{equation}
which implies that
\begin{equation}\label{eq:qn-estimates-by-Mn}
 M_n < q_n < M_n'
\end{equation}
for all $n \geqslant 0$ through simple induction.

We first note that when a number $\omega$ is $(C, \tau)$-Diophantine with $\tau = 1$\footnote{Otherwise known as a \emph{constant type number}.} then it is also Khintchine-L\'evy.
\begin{lemma}
 A number $\omega$ that is $(C, 1)$-Diophantine with some $C > 1$ satisfies $\omega \in \KL'^+(T, N)$ with $N = 1$ and $T = \log( C^{-1} + 1 ) - \kappa'$.
\end{lemma}

\begin{proof}
 By lemma \ref{lem:diophanticity-in-terms-of-continued-fraction} constant type numbers are precisely the ones with a bounded sequence of partial quotients: $a_n \leqslant C^{-1}$, which implies $M_n' < (C^{-1}+1)^n$ for all $n \geqslant 1$.
\end{proof}

Note, however, that constant type numbers form a set of measure zero (\cite{lang-diophantine}). On the other hand, the complement of the set of Diophantine numbers with fixed $\tau > 1$ and $C > 0$ is small whenever $C$ is small:
\begin{lemma}[Measure of the set of Diophantine numbers]
 The measure of the set of numbers $\omega \in [0,1]$ that are not $(C, \tau)$-Diophantine can be estimated from above by $2C \zeta(\tau)$ if $\tau>1$ and $C > 0$. Here $\zeta$ denotes the Riemann $\zeta$ function.
\end{lemma}

\begin{proof}
 The excluded numbers are contained in the set
 \begin{equation}
  \mathtt{Excl} = \bigcup_{q=1}^\infty \bigcup_{p=0}^q \left( {p \over q} - {C \over q^{1+\tau}}, {p \over q} + {C \over q^{1+\tau}} \right) \cap [0,1]
 \end{equation}
 Each of the intervals has length equal to $l(q) = 2Cq^{-(1+\tau)}$, apart from the intervals $[0, C/q^{1+\tau})$ and $(1 - C/q^{1+\tau}, 1]$ and their total length (for a fixed $q$) adds up to $ql(q)$, therefore $\lambda(\mathtt{Excl}) \leqslant \sum_{q=0}^\infty ql(q) = 2C\zeta(\tau)$.
\end{proof}

When it comes to $\tau > 1$ on the other hand it turns out that Khintchine-L\'evy numbers are Diophantine, but not the other way round.
\begin{lemma}\label{lem:kl-numbers-are-diophantine}
 If $\omega \in \KL^+(T,N)$ for some $T > 0$ and $N \in \mathbb{N}$ then $\omega$ is $(C,\tau)$-Diophantine with $C$ small enough and $\tau = 1 + {\kappa + T \over \log \varphi}$, where $\varphi = (1 + \sqrt{5})/2$. If $\omega \in \KL(T_-, T_+, N)$ for some $T_-, T_+ > 0$ and $N \in \mathbb{N}$, then it is $(C, \tau)$-Diophantine with $C$ small enough and $\tau = 1 + {T_+ + T_- \over \log \varphi}$.
\end{lemma}

\begin{proof}
 From \eqref{eq:qn-recurrence} we can infer that for any $\omega$ we have $q_n \geqslant F_{n-1}$, where $F_n$ is the Fibonacci sequence with $F_1 = F_2 = 1$, and in consequence $q_n > \varphi^n/3$. Assuming that $\omega \in \KL^+(T,N)$ we have, for all $n \geqslant N$, that
 \begin{align}\label{eq:kl-numbers-are-diophantine-estimates}
 \begin{split}
  a_{n+1} &\leqslant M_{n+1} \leqslant e^{(\kappa + T)(n+1)} = e^{\kappa + T} \left( e^{\kappa + T} \right)^n = e^{\kappa + T} \left( \varphi^n \right)^{\kappa + T \over \log \varphi} = \left( e \cdot 3^{1/\log \varphi} \right)^{\kappa + T} \cdot \left( \varphi^n/3 \right)^{\kappa + T \over \log \varphi} < \\
  &< \left( e \cdot 3^{1/\log \varphi} \right)^{\kappa + T} \cdot q_{n}^{\kappa + T \over \log \varphi}.
 \end{split}
 \end{align}
 By lemma \ref{lem:diophanticity-in-terms-of-continued-fraction} we see that $\omega$ is $(C,\tau)$-Diophantine with $\tau = 1 + {\kappa + T \over \log \varphi}$ and a suitably chosen $C$\footnote{Choosing $C$ we account for the fact that \eqref{eq:kl-numbers-are-diophantine-estimates} holds for $n \geqslant N$.}.
 
 The case of $\KL(T_-, T_+, N)$ is similar with the exception that the estimates begin with
 \begin{equation}
  a_{n+1} = {M_{n+1} \over M_n} \leqslant { e^{(\kappa+T_+)(n+1)} \over e^{(\kappa - T_-)n} } = e^{\kappa + T_+} \cdot \left( e^{(T_+ + T_-)} \right)^n
 \end{equation}
 to end with $T_+ + T_-$ instead of $\kappa + T$ and thus with $\tau = 1 + {T_+ + T_- \over \log \varphi}$.
\end{proof}

Note that in the first case in lemma \ref{lem:kl-numbers-are-diophantine} we can bring $\tau$ as close as we wish to $1 + \kappa/\log \varphi \approx 3.051$ by setting $T$ small, while in the second case the critical $\tau$ is $1$.

Using the ideas of the proof of lemma \ref{lem:kl-numbers-are-diophantine} we can infer that $\omega~\in~\KL'^+(T,N)~\cup~\KL^+(T,N)$ for some $T > 0, N \in \mathbb{N}$ implies at most exponential growth of partial quotients. Therefore any sequence of partial quotients that has a superexponential subsequence gives rise to a non-$\KL$ number $\omega$. Using this we can construct a non-$\KL$ number $\omega^*$, which is Diophantine. In fact $\omega^*$ can even have a very sparse distribution of partial quotients.

\begin{example}[A non-$\KL$ Diophantine number]
 Fix $s > 1$ and $\delta > 0$ and set $d_n := \lfloor (1+\delta)^n \rfloor$. We define $\omega^*$ through its partial quotients:
 \begin{equation}\label{eq:non-diophantine-kl-number}
  a_j = \begin{cases} 1 \mbox{ for } j \neq d_n, n = 0, 1, 2, \ldots \\ \left\lfloor e^{j^s} \right\rfloor \mbox{ for } j = d_n, n = 0, 1, 2, \ldots \end{cases}
 \end{equation}
 bearing in mind that the second case in \eqref{eq:non-diophantine-kl-number} may produce two or more values for small enough $\delta$ and $j$. For fixed $\delta$, however, there are only finitely many $j$'s for which this happens and if this is the case we define $a_j = 1$. We will not be interested in the initial partial quotients.
 
 For $s > 1$ the number $\omega$ has a superexponential subsequence of partial quotients, therefore it cannot be in $\KL'^+(T,N) \cup \KL^+(T,N)$ for any $T > 0$ and $N \in \mathbb{N}$. We will show that $\omega$ is $(C, \tau)$-Diophantine for any $\tau > A(1+\delta)^s$ and $C = C(\tau)$ small enough, where $A$ is any constant with $A > \alpha^s$ and $\alpha > 1$ is a constant specified later in the proof.\footnote{The constant $\alpha$ can be chosen as close to $1$ as we wish, at the expense of $C$. Note that this way we can make the exponent as close to the critical $\tau = 1$ as we wish.}
 
 To do this we will verify that for all $j \in \mathbb{N}$ we have
 \begin{equation}\label{eq:stronger-inequality-for-proving-diophanticity}
  a_{j+1} \leqslant C^{-1} M_j^{\tau - 1}
 \end{equation}
 as this entails \eqref{eq:diophanticity-in-terms-of-continued-fraction} and we will be able to use lemma \ref{lem:diophanticity-in-terms-of-continued-fraction}. After a minor alteration \eqref{eq:stronger-inequality-for-proving-diophanticity} is equivalent to
 \begin{equation}\label{eq:stronger-inequality-for-proving-diophanticity-altered}
  X_{j+1} \leqslant \log C^{-1} + (\tau - 1) S_j.
 \end{equation}
 Fix $j > J$ for $J$ large enough, so that there is no ambiguity in \eqref{eq:non-diophantine-kl-number} and let $n$ be such that $d_{n-1} \leqslant j < d_n$. First observe that $S_{d_{n-1}} = S_{d_{n-1} + 1} = \ldots = S_{d_n - 1}$ since $X_i = 0$ for $i = d_{n-1} + 1, \ldots, d_n - 1$. Also fix $\alpha > 1$ and note that $\lfloor x \rfloor \geqslant x/\alpha$ for $x$ large enough. Additionally set $k_0$ to be the smallest number for which $d_{k_0} \geqslant J$. We have
 \begin{align}\label{eq:Sdn-1}
 \begin{split}
  S_{d_{n-1}} &= \sum_{j = 1}^{d_{n-1}} X_j = \sum_{k=k_0}^{n-1} X_{d_k} \geqslant \sum_{k=k_0}^{n-1} (d_k^s - \log \alpha) \geqslant \sum_{k=k_0}^{n-1} ((1 + \delta)^{ks} \alpha^{-s} - \log \alpha) \geqslant \\ 
  &\geqslant \alpha^{-s} \left( \sum_{k=0}^{n-1} \left((1+\delta)^{s}\right)^k \right) - n \log \alpha = \alpha^{-s} {(1+\delta)^{ns} - 1 \over (1+\delta)^s - 1} - n \log \alpha = \\
  &= {1 \over \alpha^s((1+\delta)^s - 1)} (1+\delta)^{ns} - n \log \alpha - {1 \over \alpha^s((1+\delta)^s - 1)}.
 \end{split}
 \end{align}
 If we now prove that \eqref{eq:stronger-inequality-for-proving-diophanticity-altered} holds after we substitute $S_j$ with the right-hand side of \eqref{eq:Sdn-1} then the whole proof is concluded. To do this we need to consider two cases: $J < j < d_n - 1$ and $j = d_n - 1$. In the first case $X_{j+1} = 0$, so we need
 \begin{equation}\label{eq:non-kl-dioph-case-1}
  0 \leqslant \log C^{-1} + (\tau - 1) \left( {1 \over \alpha^s((1+\delta)^s - 1)} (1+\delta)^{ns} - n \log \alpha - {1 \over \alpha^s((1+\delta)^s - 1)} \right)
 \end{equation}
 to hold for all $n$ and some $C > 0$. This is, however, the case: the sequence in the largest brackets diverges to $+\infty$, so it must have a minimal value and we only need to set $C$ small enough to elevate the whole expression above $0$ since $\tau - 1 > 0$.
 
 The second case gives $X_{j+1} = X_{d_n} = \log \lfloor e^{d_n^s} \rfloor \leqslant d_n^s \leqslant (1+\delta)^{ns}$, we therefore similarly require
 \begin{equation}\label{eq:non-kl-dioph-case-2}
  (1+\delta)^{ns} \leqslant \log C^{-1} + (\tau - 1) \left( {1 \over \alpha^s((1+\delta)^s - 1)} (1+\delta)^{ns} - n \log \alpha - {1 \over \alpha^s((1+\delta)^s - 1)} \right).
 \end{equation}
 Subtracting $(1+\delta)^{ns}$ from both sides gives a similar inequality to \eqref{eq:non-kl-dioph-case-1}, but with a different coefficient at $(1+\delta)^{ns}$, namely
 \begin{equation}
  E := {\tau - 1 - \alpha^s((1+\delta)^s - 1) \over \alpha^s ((1+\delta)^s - 1)}.
 \end{equation}
 For $\tau > A(1+\delta)^s$ we have $E > 0$ and by an analogous argument to the one in previous case we can make inequality \eqref{eq:non-kl-dioph-case-2} valid choosing a small enough $C$.
\end{example}

\section{Measure of KL-sets: a practical point of view}\label{sec:numerical}

In this section we focus on the numerical values of estimates of theorems \ref{thm:main-theorem} and \ref{thm:main-theorem-Mn-prime} for particular values of~$T$. We outline the motivation for this in \cite{kamienski:2018-cohomological}, where we perform estimates in a small divisors problem under the assumption that the frequency $\omega$ belongs to one of the $\KL$-sets. It turns out that the quality of these estimates is best when $T$ is as small as possible. There is, however, a price to pay if we want to set $T$ small, namely we have to set $N$ large to obtain reasonable estimates on the measure of $\KL$-sets.

To better illustrate our reasoning we will focus on the set $\KL^{\prime+}(T, N)$. At the end of the section we present a detailed exposition of numerical values of estimates from theorem \ref{thm:main-theorem} for selected values of $T$ and $N$. For simplicity we will consider the case when $N = K^2$ is a square of an integer, so that the finite sum term in the estimates of theorem \ref{thm:main-theorem} vanishes.

Inequality \eqref{eq:estimates-of-the-measure-of-KL-sets-1} written for $\Xi'(T)$ tells us that the quantities that will be essential for us are the numerator 
\begin{equation}
 \mathsf{num'} : = \left( 2K + 1 + {4 \Xi'(T) \over 1 - \Xi'(T)} \right) \Xi'(T)^K
\end{equation}
and the denominator
\begin{equation}
 \mathsf{den'} := 1 - \Xi'(T)
\end{equation}
appearing on its right-hand side, our goal will be to make ${\mathsf{num'} \over \mathsf{den'}}$ as close to $0$ as possible. The numerical value of $\kappa'$ is\footnote{Analogously to a well known formula for $\kappa$ we can express $\kappa'$ as a sum of an infinite series $\kappa' = \sum_{r=1}^\infty \log_2 (r+1) \log (1 + (r(r+2))^{-1})$} $\kappa' \approx 1.410$, which suggests that it is only reasonable to consider $T$ of the same order of magnitude (and also $T < \kappa'$ when considering the set $\KL^-$\footnote{Actually even $T < \kappa' - \log 2 \approx 0.716$, since all $\omega$ satisfy $2^n \leqslant M_n'(\omega)$.}). We will therefore consider $T$ to be a number satisfying $T \leqslant 2$.
The problem is that $\Xi'$ evaluated even at a number as small as $T = 2$ is very close to $1$ and the distance to $1$ gets even smaller as we decrease $T$ towards $0$. This makes $\mathsf{den'}$ small, which tells us that $\mathsf{num'}$ needs to be even smaller. For instance
\begin{equation}
  \Xi' (2) = \exp \left( - {2 \over \left(128\left( \bar r' \right)^2 \bar\Lambda + \left(128 \bar r' \bar\Lambda \right)^{1/3} \right)^{3/2}} \right) \approx 0.9997597
\end{equation}
and this gives $\mathsf{den'}^{-1} \approx 4.161 \cdot 10^3$. The only thing we can do to overcome the effect of $\mathsf{den'}^{-1}$ being big is manipulating the exponent $K$, that appears in $\mathsf{num'}$. It turns out that, for instance, to have ${\mathsf{num'} \over \mathsf{den'}} < 10^{-2}$ we need $N \geqslant 6.084 \cdot 10^{9}$. More general numerical values are provided in table \ref{tab:KL'+} below.

Define $\mathsf{est}' = 1 - {\mathsf{num'} \over \mathsf{den'}}$. The cells in table \ref{tab:KL'+} contain the approximations of minimal values of $N$ which guarantee that the estimate $\mathsf{est}'$ is better than the value given in the leftmost column with the value of $T$ given in the top row. For instance the bottom-right cell tells us that in order to have the estimate $\mathsf{est}'$ better than $99.9\%$ with $T = 0.1$ one needs to have $N \geqslant 2.394 \cdot 10^{15}$.

\begin{table}[h]\centering
  \begin{tabular}[h]{l||m{2cm}|m{2cm}|m{2cm}|m{2cm}}
    & \centering $T=2$ & \centering $T=1$ & \centering $T=0.5$ &  $\phantom{01}\,T=0.1$ \\
  \hline \hline
  $\mathsf{est}' > \phantom{0.0}1\%$ &$\hfill 3.969\verb'e9'$ &$\hfill 8.122\verb'e10'$ &$\hfill 1.633\verb'e12'$ &$\hfill 1.610\verb'e15'$\\
  \hline
  $\mathsf{est}' > \phantom{.9}50\%$ &$\hfill 4.225\verb'e9'$ &$\hfill 8.585\verb'e10'$ &$\hfill 1.724\verb'e12'$ &$\hfill 1.681\verb'e15'$\\
  \hline
  $\mathsf{est}' > \phantom{.9}90\%$ &$\hfill 4.900\verb'e9'$ &$\hfill 9.860\verb'e10'$ &$\hfill 1.949\verb'e12'$ &$\hfill 1.854\verb'e15'$\\
  \hline
  $\mathsf{est}' > \phantom{.9}99\%$ &$\hfill 6.084\verb'e9'$ &$\hfill 1.183\verb'e11'$ &$\hfill 2.292\verb'e12'$ &$\hfill 2.116\verb'e15'$\\
  \hline
  $\mathsf{est}' > 99.9\%$ &$\hfill 7.225\verb'e9'$ &$\hfill 1.399\verb'e11'$ &$\hfill 2.663\verb'e12'$ &$\hfill 2.394\verb'e15'$\\
  
  \end{tabular}
  \vspace{0.1cm}
  \caption{Approximate minimal value of $N = K^2$ for given $T$ and desired $\mathsf{est}'$.}\label{tab:KL'+}
\end{table}

In other words the values appearing in table \ref{tab:KL'+} tell us that if we want to have a guarantee that e.g. $99.9\%$ of numbers $\omega$ satisfy the inequality
\begin{equation}
  (1 + a_1(\omega)) \ldots (1 + a_n(\omega)) \leqslant e^{(\kappa' + 0.1)n}
\end{equation}
for all $n$ ``large enough'' then ``large enough'' means ``greater than $2.394 \cdot 10^{15}$''. Observe, however, that for a given value of $T$ the entries of the table are of the same order of magnitude. This means that in order to reach a sharp measure estimate one does not pay a significantly greater price than that of crossing the threshold given by the value in the ``$\mathsf{est}' > 1\%$'' line, the ``currency'' here being the amount of initial numbers that need to be excluded from our considerations.

The values for the $\KL^+$-sets are provided in table \ref{tab:KL+}, $\mathsf{est}$ is defined analogously to $\mathsf{est'}$.
\begin{table}[h]\centering
  \begin{tabular}[h]{l||m{2cm}|m{2cm}|m{2cm}|m{2cm}}
    & \centering $T_+=2$ & \centering $T_+=1$ & \centering $T_+=0.5$ &  $\phantom{01}T_+=0.1$ \\
  \hline \hline
  $\mathsf{est} > \phantom{0.0}1\%$ &$\hfill 2.074\verb'e10'$ &$\hfill 4.238\verb'e11'$ &$\hfill 8.456\verb'e12'$ &$\hfill 8.154\verb'e15'$\\
  \hline
  $\mathsf{est} > \phantom{.9}50\%$ &$\hfill 2.220\verb'e10'$ &$\hfill 4.489\verb'e11'$ &$\hfill 8.898\verb'e12'$ &$\hfill 8.496\verb'e15'$\\
  \hline
  $\mathsf{est} > \phantom{.9}90\%$ &$\hfill 2.560\verb'e10'$ &$\hfill 5.112\verb'e11'$ &$\hfill 9.992\verb'e12'$ &$\hfill 9.328\verb'e15'$\\
  \hline
  $\mathsf{est} > \phantom{.9}99\%$ &$\hfill 3.098\verb'e10'$ &$\hfill 6.068\verb'e11'$ &$\hfill 1.166\verb'e13'$ &$\hfill 1.058\verb'e16'$\\
  \hline
  $\mathsf{est} > 99.9\%$ &$\hfill 3.686\verb'e10'$ &$\hfill 7.090\verb'e11'$ &$\hfill 1.345\verb'e13'$ &$\hfill 1.192\verb'e16'$\\
  \end{tabular}
  \vspace{0.1cm}
  \caption{Approximate minimal value of $N = K^2$ for given $T$ and desired $\mathsf{est}$.}\label{tab:KL+}
\end{table}

\section{Concluding remarks}\label{sec:conclusion}

Since theorem \ref{thm:main-theorem} holds for both $(M_n)$ and $(M_n')$ it is natural to ask whether it also does for $(q_n)$. The sequence of denominators of convergents also enjoys exponential growth almost everywhere, with rate $\ell = {\pi^2 \over 12 \log 2}$ (\cite{levy:khintchine-levy}). We were, however, not able to reproduce the reasoning of section \ref{sec:proof} due to a slightly different nature of this sequence, compared to either $(M_n)$ or $(M_n')$. The first difference between $(M_n)$ and $(q_n)$ is in the averages: we have $\mathbb{E}_\gamma \log M_n = n \kappa$, while $\mathbb{E}_\gamma \log q_n = n \ell + R_n$ with a remainder $R_n$ bounded in $n$. More importantly, however, the success of the reasoning in section \ref{sec:proof} relies on the fact that $\log M_n$ can be expressed as the sum of $n$ summands $X_j = \log a_j$, which satisfy both the mixing assumption and the Markov chain association assumption. For $\log q_n$ one can use the sequence $\left(\log \left( s_j^{-1} \right)\right)$ as a counterpart of $(X_j)$, but this sequence does not have the mixing property and this way the whole food chain of lemmas we used in \ref{sec:proof} breaks apart. To see this we need to take a closer look at the structure of the ``past'' and the ``future'' $\sigma$-algebras of the sequence $(s_n)$.\footnote{note that they are the same as the same $\sigma$-algebras for the sequence $\left(\log \left( s_j^{-1} \right) \right)$} The latter is given by $\sigma_{t}^\infty = \sigma(s_{t}, s_{t+1}, \ldots)$\footnote{By $\sigma(\ldots)$ with no indices we mean the $\sigma$-algebra generated by random variables or sets in brackets.} for $t>0$ and is thus generated by the preimages of singletons of rationals $s_j^{-1}(\{r\})$ with $j = t, t+1, \ldots$. Due to how $s_j$ is constructed from $a_1, \ldots, a_j$ the sets $s_j^{-1}(\{r\})$, however, are actually finite intersections of the preimages of singletons of positive integers through functions $a_1, \ldots, a_j$. As a consequence $\sigma_t^\infty$ can actually be written as $\sigma(a_1, a_2, \ldots)$ for any $t > 0$, which in particular means that $\sigma_t^\infty$ contains all of the ``past'' $\sigma$-algebras $\sigma_1^s = \sigma(s_1, \ldots, s_s)$ with $s < t$ as they are actually equal to $\sigma(a_1, \ldots, a_s)$ by a similar argument. This inclusion of $\sigma$-algebras is what prevents the mixing coefficients of $(s_j)$ from converging to $0$ just as was the case in example \ref{ex:psi0-infinite}, since $\sigma_1^s$ admits sets of arbitrarily small measure.

We also made a choice of sticking to $\varphi$-mixing instead of $\psi$-mixing even though we only consider quantities which are $\psi$-mixing if they exhibit any kind of mixing. This is because in the formula for $\Lambda(\cdot, \cdot)$ in definition \ref{def:Lambda} there is a dependence on $f(s,s)$ for a mixing function $f$ and an integer index $s$. In example \ref{ex:psi0-infinite} we learnt, however, that $\psi(s,s)$ may be infinite, which would yield no control over $\Lambda(\cdot, \cdot)$ and in consequence no control over the measure of $\KL$ sets. This is not the case if we consider $\varphi$-mixing.

The price we pay for this detail is, however, quite significant. The type of mixing we employ has an impact on the quality of the cumulant estimates in theorem \ref{thm:largedeviations1}. This result also has a $\psi$-mixing counterpart (\cite[Theorem 4.21, second inequality]{large-deviations-book}) in which the cumulant estimates are better - instead of a $(k!)^{2+\gamma_1}$ factor in \eqref{eq:largedeviations1-estimate} there appears $(k!)^{1+\gamma_1}$. This decrease of the exponent at $k!$, plugged into the rest of the food chain of theorems, would have an impact on theorem \ref{thm:main-theorem}.

Observe that in the proof of this theorem we sum terms of the form $e^{-\alpha\sqrt{n}}$ (with $e^{-\alpha} = \Xi(T)$) which gives a slowly converging series and large values in tables~\ref{tab:KL'+} and \ref{tab:KL+}. Using $\psi$-mixing would switch the summands to a geometric progression $e^{-\alpha n}$ whose series converges much faster and gives much better values in the counterpart of tables~\ref{tab:KL'+} and \ref{tab:KL+}. The orders of magnitudes (i.e. the exponents at $10$) in said table would reduce roughly by half. We will delve into this matter in further research of the subject as the problem seems to stem directly from the fact that we are using very general large deviation theorems which do not take into account the specifics of the very well studied sequence $(a_n)$.

Some evidence that the numbers obtained in tables \ref{tab:KL'+} and \ref{tab:KL+} are far from optimal comes also from the analysis of the continued fraction of $\pi - 3$. A (non-rigorous) analysis of its $\mathtt{range} := 4.38 \cdot 10^8$ initial partial quotients provided in \cite{bickford} gives the maximal value of $W_n(\pi) := \left| {1 \over n} S_n(\pi-3) - n\kappa \right|$ equal to $W_4(\pi) \approx 1.598$ (stemming from the unusually large $a_4(\pi) = 292$). We also have $W_n (\pi)< 0.1$ for $15 \leqslant n \leqslant \mathtt{range}$, which is inconsistent with table \ref{tab:KL+} by several orders of magnitude if we assume that $\pi - 3$ has a somewhat ``generic'' continued fraction expansion. For larger $n$ the difference is even more striking: considering $10000 \leqslant n \leqslant \mathtt{range}$ gives an oscillation of the order of magnitude of $W_n$ between $10^{-2}$ and $10^{-4}$. Using the data in \cite{bickford} and the estimates of the current paper one can only derive a rather ineffective result on $\pi - 3$ in spirit of the ones provided by tables \ref{tab:KL'+} and \ref{tab:KL+}.
\begin{corollary}
 The number $\pi - 3$ satisfies the inequality
 \begin{equation}
  M_n(\pi - 3) < 743^n
 \end{equation}
 for all $n \geqslant 1$ with probability\footnote{As in the rest of the paper by \emph{probability} we mean the Gauss measure $\gamma$.} $99.9\%$.
\end{corollary}

\begin{proof}
 The data in \cite{bickford} give the estimate $M_n(\pi - 3) \leqslant e^{(\kappa+1.598)n} < 14^n$ for $1 \leqslant n \leqslant \mathtt{range}$. For $n > \mathtt{range}$ we can use theorem \ref{thm:main-theorem} with $T = 5.62$ which gives the base $e^{\kappa + 5.62} < 743$.
\end{proof}

\section*{Acknowledgements}

During a part of the research, which led to preparation of this article, the author was supported by the Foundation for Polish Science under the MPD Programme \emph{Geometry and Topology in Physical Models}, co-financed by the EU European Regional Development Fund, Operational Program Innovative Economy 2007-2013.

\appendix

\section{Auxiliary identities}\label{sec:appendix}

\begin{lemma}[Sum of $e^{-\alpha \sqrt{n}}$]\label{lem:exp-sqrt-sum}
 For any $N \geqslant 2$ and $\alpha > 0$ the following inequality holds:
 \begin{equation}
   \sum_{n=N}^\infty e^{-\alpha \sqrt{n}} \leqslant \sum_{n=N}^{K^2-1} e^{-\alpha \sqrt{n}} + {e^{-\alpha K} \over 1 - e^{-\alpha}} \cdot \left(2K + 1 + {4 e^{-\alpha} \over 1 - e^{-\alpha}} \right),
 \end{equation}
 where $K = \left\lceil \sqrt{N} \right\rceil$.
 In particular when $N$ is a square of an integer the estimate takes form
 \begin{equation}
   \sum_{n=N}^\infty e^{-\alpha \sqrt{n}} \leqslant {e^{-\alpha K} \over 1 - e^{-\alpha}} \cdot \left(2K + 1 + {4 e^{-\alpha} \over 1 - e^{-\alpha}} \right).
 \end{equation}
\end{lemma}

\begin{proof}
 The proof relies on the following identity:
 \begin{equation}
   \sum_{s=S}^\infty (2s+1)e^{-\alpha s} = {(2S+1)e^{-\alpha S} - (2S-3)e^{-\alpha (S+1)} \over (1 - e^{-\alpha})^2}.
 \end{equation}
 We have
 \begin{align}
  \begin{split}
   \sum_{n=K^2}^\infty e^{-\alpha \sqrt{n}} &= \sum_{k=K}^\infty \sum_{n=k^2}^{(k+1)^2 - 1} e^{-\alpha \sqrt{n}} \leqslant \sum_{k=K}^\infty \sum_{n=k^2}^{(k+1)^2 - 1} e^{-\alpha k} =\sum_{k=K}^\infty (2k+1)e^{-\alpha k} =\\
   &= {(2K+1)e^{-\alpha K} - (2K-3) e^{-\alpha (K+1)} \over (1 - e^{-\alpha})^2} = {e^{-\alpha K} \over 1 - e^{-\alpha}} \cdot \left(2K + 1 + {4 e^{-\alpha} \over 1 - e^{-\alpha}} \right).
  \end{split}
 \end{align}
\end{proof}

\bibliographystyle{plain}
\bibliography{bibliography.bib}
 
\end{document}